\documentclass[12pt]{article}
\usepackage{amssymb,amsmath,amsthm,url}
\usepackage[matrix,arrow]{xy}

\newcommand{\lbr}[2]{[\hspace*{-1.5pt} [ #1 , #2 ] \hspace*{-1.5pt}]}
\newcommand{\arxiv}[1]{,\ \url{#1}}

\newtheorem{lemma}{Lemma}
\newtheorem{cor}{Corollary}
\newtheorem*{theorem}{Theorem}
\newtheorem{prop}{Proposition}
\theoremstyle{definition}
\newtheorem*{remark}{Remark}
\newtheorem*{definition}{Definition}
\title{A cohomological construction of modules \\ over
Fedosov deformation quantization algebra  }
\author{S. A. Pol'shin\thanks{E-mail: polshin.s at gmail.com}\\
{\small Institute for Theoretical Physics} \\
{\small NSC Kharkov Institute of Physics and Technology} \\
{\small Akademicheskaia St. 1, 61108 Kharkov, Ukraine }}
\date{}

\begin{document}

\maketitle

\begin{abstract}
In certain neighborhood $U$ of an arbitrary point of a symplectic
manifold $M$ we construct a Fedosov-type star-product $\ast_L$ such that for an arbitrary leaf $\wp$ of a given polarization $\mathcal{D}\subset TM$ the algebra $C^\infty (\wp \cap U)[[h]]$ has a natural structure of left module over the deformed algebra $(C^\infty (U)[[h]], \ast_L)$. With certain additional assumptions on $M$, $\ast_L$ becomes a so-called star-product with separation of variables.


\end{abstract}

\section{Introduction}

In~\cite{Fedosov1} B.V.Fedosov gave a simple construction of
deformation quantization of an arbitrary symplectic manifold (see
also~\cite{Fedosov2}). Later J.Donin~\cite{Donin} and
D.Farkas~\cite{Farkas} show the algebraic nature of Fedosov
construction. The problem of constructing modules over Fedosov
deformation quantization which generalize the states of usual
quantum mechanics  is of great interest. To this end, a notion of
adapted star-products was recently introduced~\cite{BGHHW}.
\begin{definition}
A star-product $\ast$ on $M$ is called adapted to a lagrangian
(or, more generally, coisotropic) submanifold $\wp\subset M$ if
the vanishing ideal of $\wp$ in the commutative algebra $C^\infty
(M)[[h]]$ is a left ideal in the deformed algebra
$\mathcal{A}=\left( C^\infty (M)[[h]],\ast\right)$. Then obviously
$C^\infty (\wp)[[h]]$ has a structure of left $\mathcal{A}$-module.
\end{definition}
Using the algebraic framework of~\cite{Donin,Farkas}, in the
present letter we construct a Fedosov-type star-product $\ast_L$
adapted to all the leaves of given polarization $L$ of a
symplectic manifold $M$. If $M$ is a so-called bi-Lagrangian
manifold, $\ast_L$ becomes a star-product with separation of
variables in the sense of Karabegov.

Like most of the papers on this subject, we are working in a
certain coordinate neighborhood where $\dim M$-dimensional basis
of  vector fields exists. However, only intrinsic geometric
structures affect the results, so we can try to glue star-products in
different neighborhoods together using the methods of algebraic geometry. This will be a matter of further publications. See
also~\cite{Donin,Vaisman} for the global construction of Fedosov
star-product algebras.

Plan of the present paper is the following. In Sec.~2 we consider
the Koszul complex for Weyl algebra, in Sec.~3 we define various
ideals associated with $L$, in Sec.~4 we define Fedosov complex
and prove the main result.

This article is dedicated to the memory of L.L.Vaksman with whom its preliminary versions were discussed.

\section{Koszul complex and Weyl algebra}

Let $M$ be a symplectic manifold, $\dim M=2\nu$, $U$ a certain
coordinate neighborhood  in $M$,  $A=C^\infty (U,\mathbb{C})$ a
$\mathbb{C}$-algebra of smooth functions on $U$ with pointwise
multiplication, and $E=C^\infty(U,T^\mathbb{C}M)$ a set of all
smooth complexified vector fields on $U$  with the natural
structure of an unitary $A$-module. By $T(E)$ and $S(E)$ denote
the tensor and symmetric algebra of $A$-module $E$ respectively,
and let $\wedge E^\ast$ be an algebra of smooth differential forms
on $U$. Let $\omega \in \wedge^2 E^\ast$ be a symplectic form on
$M$ and let $u:\ E\rightarrow \wedge^1 E^\ast$ be the mapping
$u(x)y=\omega (x,y), \ x,y\in E$. All the tensor products in the
present paper will be taken over $A$. Let
$$ a=x_1 \otimes \ldots \otimes x_m \otimes y_1 \wedge \ldots\wedge y_n \in
T^m (E)\otimes \wedge^n E^\ast.$$ Define the Koszul differential
of bidegree $(-1,1)$ on $T^\bullet (E)\otimes \wedge^\bullet
E^\ast$ as
$$\delta a=\sum\limits_i x_1 \otimes\ldots\otimes \hat{x}_i \otimes \ldots\otimes x_m \otimes u(x_i) \wedge y_1 \ldots \wedge y_n .$$
Let $\lambda$ be an independent variable (physically $\lambda=-i\hbar$) and $A[\lambda]=A\otimes_\mathbb{C} \mathbb{C}[\lambda]$ etc. In the
sequel we will write $A,E$ etc. instead of $A[\lambda],
A[[\lambda]],E[\lambda],$ $E[[\lambda]]$ etc. Let $\mathcal{I}_W$
be a two-sided ideal in $T(E)$ generated by relations $x \otimes
y-y \otimes x-\lambda\omega (x,y)=0$. The factor-algebra
$W(E)=T(E)/\mathcal{I}_W$ is called the Weyl algebra of $E$ and
let $\circ$ be the multiplication in $W(E)$.

A $\nu$-dimensional real distribution $\mathcal{D}\subset TM$ is
called a polarization   if it is (a) lagrangian, i.e.
$\omega(x,y)=0$ for all $x,y\in \mathcal{D}$ and (b) involutive,
i.e. $[x,y]\in \mathcal{D}$ for all $x,y\in \mathcal{D}$, where
$[.,.]$ is the commutator of vector fields on $M$. It is well
known~\cite{Etayo} that always we can choose a lagrangian
distribution $\mathcal{D}'$ transversal to $\mathcal{D}$ and let
$L$ and $L'$ be $A$-modules of smooth complexified vector fields
on $U$ tangent to $\mathcal{D}$ and $\mathcal{D}'$ respectively,
then $E=L\oplus L'$. Let $\alpha,\alpha_1,\ldots =1,\ldots,\nu$
and $\beta,\beta_1,\ldots=\nu+1,\ldots,2\nu$. Choose an $A$-basis
$\{ e_i|\, i=1,\ldots,2\nu \}$ in $E$  such that $\{ e_\alpha|\,
\alpha=1,\ldots,\nu \}$ and $\{ e_\beta|\, \beta=\nu+1,\ldots,2\nu
\}$ are the bases in $L$ and $L'$ respectively; it is always
possible in a certain coordinate neighborhood of an arbitrary point of $M$. Let $i_1,\ldots,i_p=1,\ldots,2\nu$ and
let $I=(i_1, \ldots,i_p)$  be an arbitrary sequence of indices. We
write $e_I=e_{i_1}\otimes\ldots\otimes e_{i_p}$ and we say that
the sequence $I$ is nonincreasing if $i_1 \geq i_2 \geq \ldots\geq
i_p$. We consider $\{\emptyset \}$ as a nonincreasing sequence and
$e_{ \{ \emptyset \} }=1$. We say that a sequence $I$ is of
$\alpha$-length $n$ if it contains $n$ elements less or equal
$\nu$. Let $\Upsilon^n$ be a set of all nonincreasing sequences of
$\alpha$-length $n$ and $\Upsilon_n =\bigcup_{p=n}^\infty
\Upsilon^p$. Then a variant of Poincare-Birkhoff-Witt theorem
holds~\cite{Bourbaki,Sridharan}.

\begin{theorem}[Poincare-Birkhoff-Witt]
Let $\tilde{S}(E)$ be  an $A$-submodule of $T(E)$ generated by elements $\{ e_I |\, I\in \Upsilon_0\}$. Then

(a) The restrictions $\mu_S |\tilde{S}(E)$ and
$\mu_W |\tilde{S}(E)$ of the canonical homomorphisms $\mu_S :\, T(E)\rightarrow S(E)$ and $\mu_W :\, T(E)\rightarrow W(E)$
are $A$-module isomorphisms.

(b) $\{ \mu_S (e_I) |\, I\in \Upsilon_0\}$ and
$\{ \mu_W (e_I) |\, I\in \Upsilon_0\}$ are $A$-bases of $S(E)$
and $W(E)$ respectively.

(c) $T(E)=\tilde{S}(E)\oplus \mathcal{I}_W$.
\end{theorem}

\begin{prop}
The choice of  bases in  $L$ and $L'$ does not affect the resulting isomorphism $W(E)\stackrel{\cong}{\rightarrow} S(E)$.
\end{prop}
\begin{proof}
Let $\{ e'_{i}=A_{i}^j e_j\}$ be a new basis in $E$ such that $A_{\alpha}^\beta=A^\alpha_{\beta}=0$  and let $\tilde{S}'(E)$ be a submodule in $T(E)$ generated by $ \{ e'_{I}|\, I\in \Upsilon_0\}$.
Since both $L$ and $L'$ are lagrangian, we see that for any element $a'\in \tilde{S}'(E)$ an element $a\in \tilde{S}(E)$ there exists such that $\mu_W (a)=\mu_W (a')$ and $\mu_S (a)=\mu_S (a')$. Due to Theorem~1(c) such an element is unique and the map $a'\mapsto a$ is an isomorphism.
\end{proof}

Let $\iota_m \ (m=1,2)$ be a natural embedding of $m$th direct summand
in the rhs of Theorem~1~(c) into $T(E)$, so $\mu_{S,W}|\tilde{S}(E)=\mu_{S,W}\iota_1$. Then from Theorem~1~(c) it follows that a short exact sequence of $A$-modules
$$\xymatrix@1{
0 \ar[r] & \mathcal{I}_W \ar[r]^{\iota_2} &
 T(E) \ar[r]^{\mu_W} & W(E) \ar[r] & 0
}
$$
splits, then we have another short exact sequence of $A$-modules
\begin{equation}\label{split}
\xymatrix@1@C+9pt{ 0 \ar[r] & \mathcal{I}_W \otimes \wedge E^*
\ar[r]^(.48){\iota_2 \otimes\mathop{{\rm id}}} & T(E) \otimes
\wedge E^* \ar[r]^{\mu_W \otimes \mathop{{\rm id}}} & W(E) \otimes
\wedge E^* \ar[r] & 0 }
\end{equation}
and $\iota_1\otimes\mathop{{\rm id}}$ is a natural embedding of
$\tilde{S}(E)\otimes\wedge E^*$ into $T(E)\otimes\wedge E^*$.

It is easily seen that $\delta$ preserves $\mathcal{I}_W \otimes
\wedge E^\ast$, so it induces a well-defined differential on
$W(E)\otimes\wedge E^\ast$ due to~(\ref{split}). It is well known
that $u$ is an isomorphism due to  nondegeneracy of $\omega$. So
we can define the so-called contracting homotopy of bidegree
$(1,-1)$ on $S^\bullet (E)\otimes\wedge^\bullet E^\ast$ which to
an element
$$a=x_1 \odot\ldots\odot x_m \otimes y_1 \wedge\ldots\wedge y_n
\in S^m (E)\otimes \wedge^n E^*,$$
where $\odot$ is the multiplication in $S(E)$, assigns the element
$$\delta^{-1} a=\frac{1}{m+n}\sum\limits_i (-1)^{i-1} u^{-1} (y_i)\odot x_1 \odot\ldots\odot x_m \otimes y_1                                                                 \wedge\ldots\wedge \hat{y}_i \wedge\ldots\wedge y_n $$
at $m+n>0$ and $\delta^{-1} a=0$ at $m=n=0$.

Let $a=\sum\limits_{m,n\geq 0} a_{mn}$, where $a_{mn}\in S^m (E)\otimes \wedge^n E^*$ and $\tau:\ a\mapsto a_{00}$ is the projection onto a component of bidegree $(0,0)$. Carry $\delta$ onto $S(E)\otimes \wedge E^*$ using the canonical homomorphism $\mu_S \otimes \mathop{{\rm id}}$. Then it is well known that the following equality
\begin{equation}\label{eq1}
\delta\delta^{-1}+\delta^{-1}\delta+\tau=\mathop{Id}
\end{equation}
holds. Carry the grading of $S(E)$ onto $W(E)$ using the
isomorphism $S(E)\cong W(E)$,  then $W^1 (E)\cong E$ and we will
identify them. It is easily seen that $\delta$ preserves
$\tilde{S}(E)\otimes \wedge E^*$,  so each arrow of the following
commutative diagram of $A$-modules commutes with $\delta$.
$$\xymatrix{
 & T(E)\otimes \wedge E^* \ar[ddl]_{\mu_S \otimes \mathop{{\rm id}}}
\ar[ddr]^{\mu_W \otimes \mathop{{\rm id}}} \\
 & \tilde{S}(E)\otimes \wedge E^* \ar[u]|{\iota_1 \otimes \mathop{{\rm id}}}
\ar[dl]^{\mu_S \iota_1 \otimes \mathop{{\rm id}}}|(.35)\cong
\ar[dr]_{\mu_W \iota_1 \otimes \mathop{{\rm id}}}|(.35)\cong \\
S(E)\otimes \wedge E^*  && W(E)\otimes \wedge E^*. }$$ Then
$\delta$ commutes with $A$-module isomorphism $\mu_W \iota_1
(\mu_S \iota_1)^{-1}\otimes \mathop{{\rm id}}$. Carry the
contracting homotopy $\delta^{-1}$ and the projection $\tau$ from
$S(E)\otimes\wedge E^*$ onto $W(E)\otimes\wedge E^*$ via this
isomorphism, then the equality~(\ref{eq1}) remains true. Let
$\delta W^\bullet =(W(E)\otimes \wedge^n E^*, \ \delta)$, then
from~(\ref{eq1}) it follows that
\begin{equation}\label{eq4} H^0 (\delta W^\bullet)=A, \qquad H^n
(\delta W^\bullet)=0, \  n>0.
\end{equation}

\section{The ideals}

Let $\mathcal{I}_\wedge$ be an ideal in $\wedge E^*$ those elemens annihilate the polarization $L$, i.e. $\mathcal{I}_\wedge=\sum\nolimits_{n=1}^\infty
\mathcal{I}_\wedge^n$, where
$$\mathcal{I}_\wedge^n =\{ \alpha\in \wedge^n E^* |\, \alpha (x_1,\ldots,x_n)=0
\ \forall x_1,\ldots,x_n \in L \}.$$ It is well known that locally
$\mathcal{I}_\wedge$ is generated by $\nu$ independent 1-forms
which are the basis of $\mathcal{I}_\wedge^1$. On the other hand,
$L$ is lagrangian, so from the dimensional reasons we obtain
$u(L)=\mathcal{I}_\wedge^1$, so
\begin{equation}\label{eq5}
\mathcal{I}_\wedge=(u(L)).
\end{equation}

Let $\mathcal{I}_L$ be a left ideal in $W(E)$ generated by elements of $L$ and $\mathcal{I}=\mathcal{I}_L\otimes\wedge
E^* +W(E)\otimes\mathcal{I}_\wedge$  a left ideal in
$W(E)\otimes\wedge E^*$. Then from~(\ref{eq5}) it follows that
\begin{equation}\label{eq6}
\delta(\mathcal{I})\subset \mathcal{I}.
\end{equation}

Let $\mathbb{N}_0=\mathbb{N}\cup \{0\}$. A semigroup $(S,\vee)$ is
called \textit{filtered} if a decreasing filtration $S_i,\ i\in
\mathbb{N}_0$ on $S$ there exists such that $S_0 =S$ and $S_i \vee
S_j \subset S_{i+j}\ \forall i,j$. Let $I,J\in \Upsilon_0$,
$I=(i_1,\ldots,i_m)$, $J=(j_i,\ldots,j_n)$ and let $I\vee J$ be
the set $\{ i_1,\ldots,i_m, j_i,\ldots,j_n \}$ arranged in the
descent order. Then $(\Upsilon_0,\vee)$ becomes a semigroup
filtered by $\Upsilon_i$.

\begin{lemma}\label{lem1}
Let $\mathcal{I}^{(S)}_L$ be an ideal  in $S(E)$ generated by
elements of $L$, then $\mu_W \iota_1 (\mu_S
\iota_1)^{-1}\mathcal{I}^{(S)}_L=\mathcal{I}_L$.
\end{lemma}
\begin{proof}
Since $L$ is lagrangian, we have $e_{\alpha_1}\circ
e_{\alpha_2}=e_{\alpha_2}\circ e_{\alpha_1}$
$\forall\alpha_1,\alpha_2$, thus for any $I\in\Upsilon_0$ we have
$\mu(e_I)\circ e_\alpha=\mu(e_{I\vee\{\alpha\}})$ and
$I\vee\{\alpha\}\in \Upsilon_1$. Then from Theorem~1~(b) it
follows $\mathcal{I}_L  \subset\mathop{{\rm span}}_A \{ \mu_W
(e_I) |\, I\in \Upsilon_1\}$. On the other hand, if $I=(i_1
,\ldots,i_p)\in \Upsilon_1$ then $1\leq i_p \leq n$, so $\mu_W
(e_I)\in\mathcal{I}_L$. Then $\mathop{{\rm span}}_A \{ \mu_W (e_I)
|\, I\in \Upsilon_1\}\subset \mathcal{I}_L$  and we obtain
 $\mathcal{I}_L=\mu_W \iota_1 (\tilde{S}_1 (E))$,
where $\tilde{S}_i (E)= \mathop{{\rm span}}_A \{ e_I |\, I\in
\Upsilon_i\},\ i\in\mathbb{N}_0$ is a decreasing filtration on
$\tilde{S}(E)$. Analogously $\mathcal{I}_L^{(S)}=\mu_S \iota_1
(\tilde{S}_1 (E))$, which proves the lemma.
\end{proof}
From~(\ref{eq5}) it is easily seen that $\delta^{-1}$ preserves the submodule $\mathcal{I}^{(S)}_L\otimes\wedge E^*
+S(E)\otimes\mathcal{I}_\wedge$ of $S(E)\otimes\wedge E^*$,
then using Lemma~\ref{lem1}  we obtain
\begin{equation}\label{eq7}
\delta^{-1}(\mathcal{I})\subset \mathcal{I}.
\end{equation}
\begin{remark}
The choice of $\tilde{S}(E)$ in Theorem~1 is crucial for our
construction of contracting homotopy of $\delta W^\bullet$. The usual choice of submodule $S'(E)$ of symmetric tensors in $T(E)$ instead
of $\tilde{S}(E)$ yields another contracting homotopy of
$\delta W^\bullet$ which does not preserve $\mathcal{I}$.
\end{remark}
Let $\nabla$ be an exterior derivative on $\wedge E^*$ which to an element $\alpha\in \wedge^{n-1} E^*$ assigns the element
\begin{equation}\label{eq8}
\begin{split}
(\nabla\alpha)(x_1,\ldots,x_n)=
\sum\limits_{1\leq i<j\leq n}
(-1)^{i+j} \alpha([x_i,x_j],x_1,\ldots,\hat{x}_i,\ldots,\hat{x}_j,\ldots x_n)
\\ +\sum\limits_{1\leq i\leq n} (-1)^{i-1}x_i \alpha
(x_1,\ldots,\hat{x}_i,\ldots, x_n).
\end{split}
 \end{equation}
Let $\nabla_x y\in E,\ x,y\in E$ be a connection on $M$, then we
can extend $\nabla_x$ to $T(E)$ by the Leibniz rule. It is well
known that we can always choose $\nabla$ in such a way that
$\nabla_x \omega=0\ \forall x\in E$. It is well known that such a
connection preserve $\mathcal{I}_W$ for all $x\in E$, so it
induces a well-defined derivation on $W(E)$. Now consider $\nabla$
as a map $W(E)\rightarrow W(E)\otimes \wedge^1 E^*$ such that
$(\nabla a)(x)=\nabla_x a$. Then it is well known that $\nabla$
may be extended to a $\mathbb{C}[[\lambda]]$-linear derivation of
bidegree $(0,1)$ of the whole algebra
$W^\bullet(E)\otimes\wedge^\bullet E^*$ whose restriction to
$\wedge E^*$ coincides with~(\ref{eq8})

\begin{lemma}\label{lem2}
Let $a\in W(E)\otimes\wedge^n E^*$ and $a(x_1,\ldots,x_n)\in V$ for all $x_1,\ldots,x_n \in L$, where $V$ is a submodule of $W(E)$. Then
$a\in V\otimes\wedge E^* +W(E)\otimes\mathcal{I}_\wedge$.
\end{lemma}
\begin{proof}
Let $\{ \tilde{e}\vphantom{e}^i|\, i=1,\ldots,2\nu\}$ be a basis
of $E^*$ dual to $\{ e_i \}$, i.e. $\tilde{e}\vphantom{e}^i
(e_j)=\delta^i_j$. Then from the well-known theorem of basic
algebra it follows that an arbitrary $a\in W(E)\otimes\wedge^n
E^*$ may be represented in the form
$a=\sum\limits_{i_1<\ldots<i_n} a_{i_1\ldots i_n} \otimes
\tilde{e}\vphantom{e}^{i_1}\wedge\ldots\wedge
\tilde{e}\vphantom{e}^{i_n},$ where $a_{i_1\ldots i_n}\in W(E)$
for all $i_1,\ldots,i_n$. It is easily seen that $\{
\tilde{e}\vphantom{e}^\beta |\, \beta=\nu+1,\ldots,2\nu \}$
generate $\mathcal{I}_\wedge$, so
$$a=\sum\limits_{\alpha_1<\ldots<\alpha_n} a_{\alpha_1\ldots \alpha_n} \otimes
\tilde{e}\vphantom{e}^{\alpha_1}\wedge\ldots\wedge
\tilde{e}\vphantom{e}^{\alpha_n}+W(E)\otimes \mathcal{I}^n_\wedge.$$
On the other hand, $a_{\alpha_1\ldots \alpha_n}
=a(e_{\alpha_1},\ldots,e_{\alpha_n})\in V$ due to the lemma's conditions. So $a\in V\otimes\wedge E^* +W(E)\otimes\mathcal{I}_\wedge$.
 \end{proof}

We say that a polarization (or, more generally, distribution)
$\mathcal{D}$ is self-parallel wrt $\nabla$ iff
\begin{equation}\label{eq9}
\nabla_x y\in L, \quad x,y\in L.
\end{equation}
For a given $\mathcal{D}$, a torsion-free connection which
obeys~(\ref{eq9}) always exists (\cite{BF}, Theorem~5.1.12).
Proceeding along the same lines as in the proof of \cite{Xu},
Lemma 5.6, we obtain another torsion-free connection
$\tilde{\nabla}$ on $M$ such that $\tilde{\nabla}_x \omega=0\
\forall x\in E$ and $\mathcal{D}$ is self-parallel wrt
$\tilde{\nabla}$. Suppose $\nabla$ is a connection (not
necessarily torsion-free) such that $\nabla_x \omega=0\ \forall
x\in E$ and $\mathcal{D}$ is self-parallel wrt $\nabla$. Then
$\nabla_x \mathcal{I}_L\subset \mathcal{I}_L \ \forall x\in L$, so
using Lemma~\ref{lem2} we obtain $\nabla\mathcal{I}_L\subset
\mathcal{I}$. On the other hand, the involutivity of $L$ together
with~(\ref{eq8}) yield $\nabla\mathcal{I}_\wedge
\subset\mathcal{I}_\wedge$ (Frobenius theorem), so we finally
obtain
\begin{equation}\label{eq10}
\nabla\mathcal{I}\subset\mathcal{I}.
\end{equation}
Let $\mathcal{R}\in E\otimes E^* \otimes\wedge^2 E^*$ be the
curvature tensor of $\nabla$, i.e.
$$\mathcal{R}(x,y)z=\nabla_x \nabla_y z-\nabla_y \nabla_x z-\nabla_{[x,y]}z, \qquad x,y,z\in E.$$
Then using~(\ref{eq9}) and the involutivity
of $L$  we obtain $\mathcal{R}(x,y)z\in L \quad \forall x,y,z \in
L$. Then Lemma~\ref{lem2} yields $\mathcal{R}(x,y)\in \mathcal{I}
\quad \forall x,y \in L$. On the other hand,
$\mathcal{R}(x,y)\in W^1 (E)\otimes \wedge^1 E^*$, so an element
$R\in W^2 (E)\otimes \wedge^2 E^*$ there exists such that
$R(x,y)=\delta^{-1} (\mathcal{R}(x,y)) \quad \forall x,y\in E$.
Using~(\ref{eq7}) we see that $R(x,y)\in \mathcal{I}_L \ \forall
x,y \in L$ and using Lemma~\ref{lem2} we obtain
\begin{equation}\label{eq11}
R\in \mathcal{I}.
\end{equation}
Let $R^i_{\ jkl}=\tilde{e}\vphantom{e}^i(\mathcal{R}(e_k,e_l)e_j)$
be the components of curvature tensor. Denote $R^{ij}_{\ \
kl}=\omega^{jp} R^i_{\ pkl}$, where $\omega^{ij}$ is the matrix
inverse to $\omega(e_i,e_j)$. Then it is easily seen that the difference between our $R$ and the one introduced in~\cite{Fedosov1,Fedosov2} belongs to the center of $W(E)\otimes \wedge E^*$, so we obtain
$$\nabla^2 a=\frac{1}{\lambda} \lbr{R}{a} \qquad \forall a\in W(E)\otimes\wedge E^* ,$$
where $\lbr{\cdot}{\cdot}$ is the commutator in $W(E)\otimes\wedge E^*$.

Let $T\in W^1 (E)\otimes \wedge^2 E^*$, $T(x,y)=\nabla_x y-\nabla_y x-[x,y]$ be the torsion of  $\nabla$.
Using~(\ref{eq9}) and the involutivity of $\mathcal{D}$ we see that $T(x,y)\in L \ \forall
x,y \in L$, then using Lemma~\ref{lem2} we obtain
\begin{equation}\label{eq11'}
T\in \mathcal{I}.
\end{equation}

Suppose $\wp$ is a leaf of the distribution $\mathcal{D}$ such
that $\wp\cap U\not=\emptyset$, $\Phi=\{ f\in A|\ f|\wp=0\}$ is
the vanishing ideal of $\wp$ in $A$, $\mathcal{I}_\Phi$ is an
ideal in $W(E)\otimes\wedge E^*$ generated by elements of $\Phi$,
and $\mathcal{I}_{\text{fin}}=\mathcal{I}+\mathcal{I}_\Phi$ is a
homogeneous ideal in $W(E)\otimes\wedge E^*$. Then due
to~(\ref{eq6}),(\ref{eq7}) we can define the subcomplex $\delta
\mathcal{I}_{\text{fin}}^\bullet=(\mathcal{I}_{\text{fin}},\delta)$ with the same contracting homotopy $\delta^{-1}$. Note that $\tau
(\mathcal{I}_{\text{fin}})=\Phi$, then using~(\ref{eq1}) we obtain
\begin{equation}\label{eq13}
H^0 (\delta\mathcal{I}_{\text{fin}}^\bullet)=\Phi, \qquad H^n
(\delta\mathcal{I}_{\text{fin}}^\bullet)=0,\ n>0
\end{equation}

It is easily seen that vector fields of $L$
preserve $\Phi$, i.e. $(\nabla f)(x)\in\Phi \quad \forall f\in \Phi, x\in L$. Then from Lemma~\ref{lem2} we obtain $\nabla\Phi\in\mathcal{I}_\Phi+ \mathcal{I}^1_\wedge$, so we finally obtain
\begin{equation}\label{eq11a}
\nabla\mathcal{I}_\Phi \subset\mathcal{I}_{\text{fin}}.
\end{equation}

\section{Fedosov complex and star-product}

Let $W^{(i)}(E)$ be the grading in $W(E)$ which coincides with
$W^i (E)$ except for the $\lambda\in W^{(2)}(E)$, and let
$W_{(i)}(E)$ be a decreasing filtration generated by $W^{(i)}(E)$.
Suppose $\widehat{W}(E)$, $\widehat{\mathcal{I}}$ are completions
of $W(E)$, $\mathcal{I}$ with respect to this filtration, then
$\widehat{\mathcal{I}}$ is a left ideal in $\widehat{W}(E)\otimes
\wedge E^*$.
 Let $A_i,\ i\in\mathbb{N}_0$ be an $(\lambda)$-adic
filtration in $A$, then $\tau (W_{(i)}(E))\subset A_{\{ i/2\}}$.
Then $W_{(2i)}(E)\subset \tau^{-1} (A_{i})$, so $\tau$ is
continuous in the topologies generated by $W_{(i)}(E)$ and $A_{i}$
and thus can be extended to a mapping $\widehat{W}(E)\rightarrow
\widehat{A}$. Since $\delta,\delta^{-1}$ and $\nabla$ have fixed
bidegrees with respect to bigrading $W^{(i)}(E)\otimes\wedge^n
E^*$,  we can extend them to derivations of $\widehat{W}(E)\otimes
\wedge E^*$ in such a way that Eq.~(\ref{eq1}) remains true and
they preserve $\widehat{\mathcal{I}}$. So we will write $A,W(E)$
etc. instead of $\widehat{A}$, $\widehat{W}(E)$ etc.

Let
$$r_0=\delta^{-1}T, \qquad r_{n+1}=\delta^{-1}\left(R+\nabla r_n +
\frac{1}{\lambda}r^2_n \right), \quad n\in\mathbb{N}_0$$ Then it
is well known that the sequence $\{r_n\}$ has a limit $r\in
W_{(2)}(E)\otimes \wedge^1 E^*$. Then we can define well-known
Fedosov complex $DW^\bullet=(W(E)\otimes \wedge^n E^*, \ D)$ with
the differential (\cite{Fedosov1,Fedosov2}, see
also~\cite{Donin,KS} for the case of nonzero torsion)
$$D=-\delta+\nabla+\frac{1}{\lambda} \lbr{r}{\cdot}.$$
Let $F$ be an Abelian group which is complete with respect to its
decreasing filtration $F_i ,\ i\in\mathbb{N}_0$, $\cup F_i=F$,
$\cap F_i=\emptyset$. Let $\deg a=\max \{i: a\in F_i\}$ for $a\in
F$.
\begin{lemma}[\cite{Donin}]\label{lem3}
 Let $\varphi:\ F\rightarrow F$ be a set-theoretic map such that
$\deg (\varphi(a)-\varphi(b))>\deg (a-b)$ for all $a,b\in F$. Then the map $Id+\varphi$ is invertible.
\end{lemma}
Let $Q:\ W(E)\otimes\wedge E^* \rightarrow W(E)\otimes\wedge E^*$,
$Q=Id+\delta^{-1}(D-\delta)$ be a $\mathbb{C}[[\lambda]]$-linear
map, then it is well known that $\delta Q=QD$ and from
Lemma~\ref{lem3} it follows that $Q$ is invertible, so it is a
chain equivalence and we obtain
$$ H^n (Q):\ H^n(DW^\bullet)\cong_{\mathbb{C}[[\lambda]]} H^n (\delta W^\bullet), \quad
n\in\mathbb{N}_0 $$
Using~(\ref{eq7}),(\ref{eq10}),(\ref{eq11}),(\ref{eq11'}) and
taking into account that $\mathcal{I}$ is a left ideal in
$W(E)\otimes\wedge E^*$ we have $r_n\in \mathcal{I}$ for all $n$,
so $r\in \mathcal{I}$.
Using~(\ref{eq6}),(\ref{eq7}),(\ref{eq10}),(\ref{eq11a}) we see
that $D\mathcal{I}_{\text{fin}} \subset \mathcal{I}_{\text{fin}}$
and $Q\mathcal{I}_{\text{fin}} \subset \mathcal{I}_{\text{fin}}$,
so we can define the subcomplex
$D{\mathcal{I}}_{\text{fin}}^\bullet=({\mathcal{I}}_{\text{fin}},D)$
and using Lemma~\ref{lem3} we obtain
$$ H^n (Q):\ H^n (D\mathcal{I}_{\text{fin}}^\bullet)\cong_{\mathbb{C}[[\lambda]]}
 H^n (\delta\mathcal{I}_{\text{fin}}^\bullet),\quad n\in\mathbb{N}_0 . $$
Then due to~(\ref{eq4}),(\ref{eq13}) we have the following diagram
\begin{equation}\label{cd-fin}
\xymatrix{
A=H^0 (\delta W^\bullet)\ar[r]_(.5)\cong^(.55) {Q^{-1}} & H^0 (DW^\bullet) \\
\Phi=H^0 (\delta\mathcal{I}_{\text{fin}}^\bullet)
\ar@<-2.35ex>[u]_\bigcup \ar[r]_(.5)\cong^(.55) {Q^{-1}} & H^0
(D\mathcal{I}_{\text{fin}}^\bullet).\ar[u]_\bigcup }
\end{equation}
Then we can define the Fedosov-type star-product $A\times A\ni
(f,g)\mapsto f\ast_L g\in A$ on $U$ carrying the multiplication
from $H^0 (DW^\bullet)$ onto $H^0 (\delta W^\bullet)$:
 \begin{equation}\label{ast}
f\ast_L g=Q(Q^{-1}f \circ Q^{-1}g).
\end{equation}
Then from~(\ref{cd-fin}) we see that $\Phi$ is a left ideal in
$\mathcal{A}_L (U)=(A,\ast_L)$ since $H^0
(D\mathcal{I}_{\text{fin}}^\bullet)$ is a left ideal in $H^0
(DW^\bullet)$. Since the choice of $\wp$ does not affects
$\ast_L$, we have proved the following result.
\begin{prop}
Let $M$ be a symplectic manifold and let $\mathcal{D}\subset TM$
be a real polarization on $M$.  Then in a certain coordinate neighborhood of an arbitrary point of $M$  we can construct a star-product  adapted to all the leaves of $\mathcal{D}$ which depends on the intrinsic geometric structures on $M$ only.
 \end{prop}
This extends the results of Reshtikhin and Yakimov~\cite{RY}, Xu~\cite{Xu}, and Donin~\cite{BD}
  who constructed  commutative subalgebras of
Fedosov deformation quantization algebra associated to a Lagrangian fiber bundle, lagrangian submanifold and polarization respectively.
\begin{cor}\label{cor1}
$A/\Phi\cong C^\infty (\wp\cap U)$ has a natural structure of left
$\mathcal{A}_L (U)$-module.
\end{cor}
This extends the results of Bordemann, Neumaier and Waldmann~\cite{cot} who constructed the modules over Fedosov deformation quantization of cotangent bundles.

Suppose $\mathcal{D}'$ is a polarization, then $M$ is a
\textit{bi-Lagrangian manifold}~\cite{Etayo} (called the Fedosov
manifold of Wick type in~\cite{Dolg}) and we can  choose a
connection on $M$ such that $\nabla_z x\in L$ and $\nabla_z y\in
L'$, $x\in L,\ y\in L', \ z\in E$. It is well known that for any
two transversal involutive distributions
$\mathcal{D},\mathcal{D}'$ we can choose a coordinate system $\{
x_i\}$ in a certain neighborhood $U$ of any point of $M$ such that
$\{ \partial/\partial x_\alpha |\, \alpha=1,\ldots,\nu\}$ and $\{
\partial/\partial x_\beta|\, \beta=\nu+1,\ldots,2\nu\}$ are local
bases in $L$ and $L'$ respectively (\cite{BC}, Ch.1, Problem~30;
see also~\cite{W}, Sec.4.9 for the case when
$\mathcal{D},\mathcal{D}'$ are polarizations). Let $\wp'$ be a
leaf of $\mathcal{D}'$ such that $\wp'\cap U\not=\emptyset$ and
$\Phi'=\{ f\in A|\ f|\wp'=0\}$.
 Then analogously to Proposition~2 we see that $\Phi'$ is a right ideal in $\mathcal{A}_L (U)$. Write the star-product~(\ref{ast}) as a bidifferential operator on $U$:
t$$f\ast_L g=\sum\limits_{r,s}\sum\limits_{i_1,\ldots,i_r\atop j_1,\ldots,j_s}
\Lambda^{i_1\ldots i_r |j_1 \ldots j_s} \frac{\partial^r
f}{\partial x^{i_1}\ldots \partial x^{i_r}} \frac{\partial^s
g}{\partial x^{j_1}\ldots \partial x^{j_s}}.$$ Since $\Phi$ is a
left ideal in $\mathcal{A}_L (U)$, we see that $\Lambda^{i_1\ldots
i_r| \beta_1 \ldots \beta_s}\in\Phi$ and analogously
$\Lambda^{\alpha_1\ldots\alpha_r|j_1\ldots j_s}\in \Phi'$ since
$\Phi'$ is a right ideal. But $\wp,\wp'$ are arbitrary, so
$\ast_L$ is a star-product with separation of variables in the
sense of Karabegov~\cite{Kar96} (called a star-product of Wick
type in~\cite{BW}):
$$f\ast_L g=\sum\limits_{r,s} \sum\limits_{\beta_1\ldots\beta_r \atop
\alpha_1\ldots\alpha_s}
\Lambda^{\beta_1\ldots\beta_r|\alpha_1\ldots\alpha_s}
\frac{\partial^r f}{\partial x^{\beta_1}\ldots \partial
x^{\beta_r}} \frac{\partial^s g}{\partial x^{\alpha_1}\ldots
\partial x^{\alpha_s}}.$$
This extends the results of Bordemann and
Waldman~\cite{BW} who constructed a Fedosov star-product of Wick
type on arbitrary K\"ahler manifold.

\end{document}